		\documentclass[10pt]{elsarticle}
		\usepackage{etoolbox}
		\makeatletter
		\patchcmd{\ps@pprintTitle}{\footnotesize\itshape
		Preprint submitted to \ifx\@journal\@empty Elsevier
		\else\@journal\fi\hfill\today}{\relax}{}{}
		\makeatother
\usepackage{lineno,hyperref}
\modulolinenumbers[5]
\usepackage{amsfonts}
\usepackage{amssymb}
\usepackage{amsmath}
\usepackage{amsthm}
\usepackage{graphicx}
\usepackage{stackrel}
\usepackage{filecontents}
\usepackage{enumerate}
\newtheorem{theorem}{Theorem}[section]
\newtheorem{lemma}{Lemma}[section]
\newtheorem{rem}{Remark}[section]
\journal{arXiv}
\def\tr{{\rm{tr}}}

\begin{document}

\begin{frontmatter}
\title{The Basins of Attraction in a Modified May--Holling--Tanner Predator-Prey Model with Allee Effect}
\author[QUTaddress,UDLAaddress]{Claudio Arancibia--Ibarra}
\address[QUTaddress]{School of Mathematical Sciences, Queensland University of Technology (QUT), Brisbane, Australia.}
\address[UDLAaddress]{Facultad de Educaci\'on, Universidad de Las Am\'ericas (UDLA), Santiago, Chile.}
\begin{abstract}
I analyse a modified May--Holling--Tanner predator-prey model considering an Allee effect in the prey and alternative food sources for predator. Additionally, the predation functional response or predation consumption rate is linear. The extended model exhibits rich dynamics and we prove the existence of separatrices in the phase plane separating basins of attraction related to oscillation, co-existence and extinction of the predator-prey population. We also show the existence of a homoclinic curve that degenerates to form a limit cycle and discuss numerous potential bifurcations such as saddle-node, Hopf, and Bogadonov--Takens bifurcations. We use simulations to illustrate the behaviour of the model.
\end{abstract}
\end{frontmatter}

\newpage
\section{Introduction}
In the last decade interactions between species are appearing in different fields of Population Dynamics. In particular, predation models have proposed and studied extensively due to their increasing importance in both Biology and applied Mathematics \cite{yu, zhao}. Predation models such as the Holling-Tanner model \cite{arrows} (or May--Holling--Tanner \cite{saez}) are of particular mathematical interest for temporal and spatio--temporal domains \cite{banerjee1, ghazaryan, martinez}. The Holling--Tanner model has been used extensively to model many real--world predator--prey interactions \cite{hanski,hanski2,hanski3,turchin2}. For instance, it has been used by Hanski {\em et al.\ } \cite{hanski} to investigate the predator--prey interaction between the least weasel (Mustela nivalis) and the field vole (Microtus agrestis). This study was based  under the hypothesis that generalist predators predicts correctly the geographic gradient in rodent oscillations in Fennoscandia. Additionally, the authors showed that the amplitude and cycle period decreasing from north to south.
	
The traditional Holling--Tanner model is describing by the following pair of equations 
	\begin{equation}\label{prey-predator}
	\begin{aligned}
	\dfrac{dx}{dt} &=	 rx\left( 1-\dfrac{x}{K}\right) - H(x)y, \\
	\dfrac{dy}{dt} &=	 sy\left( 1-\dfrac{y}{nx}\right).
	\end{aligned}
	\end{equation} 
	Here $x(t)$ is used to represent the size of the prey population at time $t$, and $y(t)$ is used to represent the size of the predator population at time $t$. Moreover, $r$ is the intrinsic growth rate for the prey, $s$ is the intrinsic growth rate for the predator, $n$ is a measure of the quality of the prey as food for the predator, $K$ is the prey environmental carrying capacity, $nx$ can be interpreted as a prey dependent carrying capacity for the predator. Moreover, the growth of the predator and prey population is a logistic form and all the parameters are assumed to be positive. In addition, the behaviour of the Holling--Tanner model depends on the type of the predation functional response chosen. This response function is used to represent the impact of predation upon the prey species. A Holling Type I functional response provides a mechanism to explain the survival advantage for animals to form large groups or herds assuming protection from external threats. In many cases, clustering reduces the total area (relative to the total mass) exposed to chemicals, extreme weather, bacteria or predators \cite{ turchin, may}. In \eqref{prey-predator} the functional response corresponds to $H(x) = qx$. Where $q$ is the maximum predation rate per capita \cite{turchin}. The resulting  Holling--Tanner model is an autonomous two--dimensional system of differential equations and is given by Kolmogorov type \cite{freedman} given by
	\begin{equation}\label{ht1}
	\begin{aligned}
	\dfrac{dx}{dt} & = rx \left( 1-\dfrac{x}{K}\right)-qxy, \\ 
	\dfrac{dy}{dt} & = sy\left( 1\ -\dfrac{y}{nx}\right) \,.
	\end{aligned}  
	\end{equation}

Additional complexity can be incorporated into these models in order to make them more realistic. In the case of severe prey scarcity, some predator species can switch to another available food, although its population growth may still be limited by the fact that its preferred food is not available in abundance. For instance, least weasel is a generalist and nomadic species \cite{hanski}. This ability can be modelled by adding a positive constant $c$ to the environmental carrying capacity for the predator \cite{aziz}. Therefore, we have a modification to the logistic growth term in the predator equation, namely $(1-y/nx)$ is replaced by $(1-y/(nx+c))$ as shown below;
	
	\begin{equation}\label{mht1}
	\begin{aligned}
	\dfrac{dx}{dt} &= rx\left(1-\dfrac{x}{K} \right) - qxy\,, \\
	\dfrac{dy}{dt} &= sy\left(1-\dfrac{y}{nx+c} \right) \,.
	\end{aligned} 
	\end{equation}
	
	Models such as \eqref{mht1} are known as modified Holling--Tanner models \cite{aziz, arancibia, feng, singh} as the predator acts as a generalist since it avoids extinction by utilising an alternative source of food. Note that the Holling--Tanner predator-prey model also considers the case of a specialist predator, i.e.\ $c=0$ \cite{korobei}. It is assumed that a reduction in a predator population has a reciprocal relationship with per capita availability of its favorite food \cite{aziz}. Nevertheless, when $c>0$, the modified Holling--Tanner does not have these abnormalities and it enhances the predictions about the interactions. This model was proposed in \cite{aziz}, but the model was only analysed partially. Using a Lyapunov function \cite{korobei}, the global stability of a unique positive equilibrium point was shown. 
	
	On the other hand, in this manuscript we consider a density-dependent phenomenon in which fitness, or population growth, increases as population density increases \cite{kramer, allee, berec, courchamp, stephens}. This phenomenon is called Allee effect or positive density dependence in population dynamics \cite{liermann}.	These mechanisms are connected with individual cooperation such as strategies to hunt, collaboration in unfavourable abiotic conditions, and reproduction \cite{courchamp3}. When the population density is low species might have more resources and benefits. However, there are species that may suffer from a lack of conspecifics. This may impact their reproduction or reduce the probability to survive when the population volume is low \cite{courchamp2}. The Allee effect may appear due to a wide range of biological phenomena, such as reduced anti-predator vigilance, social thermo-regulation, genetic drift, mating difficulty, reduced defense against the predator, and deficient feeding because of low population densities \cite{stephens2}. With an Allee effect included, the Holling--Tanner Type I model \eqref{mht1} becomes
\begin{equation}
\begin{aligned} \label{hta1}
\dfrac{dx}{dt} &= rx\left(1-\dfrac{x}{K}\right)(x-m) - qxy\,,\\
\dfrac{dy}{dt} &= sy\left(1-\dfrac{y}{nx+c} \right).
\end{aligned}
\end{equation}
The growth function $\jmath(x) = rx(1-x/K)(x-m)$ has an enhanced growth rate as the population increases above the threshold population value $m$. If $\jmath(0)=0$ and $\jmath'(0) \geq 0$ - as it is the case with $m  \leq 0$ - then $\jmath(x)$ represents a proliferation exhibiting a weak Allee effect, whereas if $\jmath(0)=0$ and $\jmath'(0)<0$ - as it is the case with $m>0$ - then $\jmath(x)$ represents a proliferation exhibiting a strong Allee effect \cite{yue2}. 
	
The Holling--Tanner model with Allee effect is discussed further in Section~\ref{mod} and a topological equivalent model is derived. In Section~\ref{res}, we study the main properties of the model. That is, we prove the stability of the equilibrium points and give the conditions for saddle-node bifurcations and Bogadonov--Takens bifurcations. In Section~\ref{bas} we study the impact in the basins of attraction of the inclusion of the modification. We conclude the manuscript summarising the results and discussing the ecological implications.

\section{The Model}\label{mod}
The Holling--Tanner model with Allee effect and alternative food is given by \eqref{hta1}, and for biological reasons we only consider the model in the domain $\Omega=\{(x,y)\in\mathbb{R}^2, ~x\geq0, ~y\geq0\}$. The equilibrium points of system \eqref{hta1} are $(K,0)$, $(m,0)$, $(0,c)$, and $(x^*,y^*)$, this last point(s) being defined by the intersection of the nullclines $y=nx+c$ and $y=r(1-x/K)(x-m)/q$.  In order to simplify the analysis, we follow \cite{arancibia,arancibia2,blows} and convert \eqref{hta1} to a topologically equivalent nondimensionalised model that has fewer parameters. Following \cite{arancibia,arancibia2,blows} we introduce a change of variable and time rescaling, given by the function $\varphi :\breve{\Omega}\times\mathbb{R}\rightarrow \Omega\times\mathbb{R}$, where $\varphi(u,v,\tau)=(x,y,t)$ is defined by $x=Ku$, $y=Knv$, $d\tau =Krdt$ and $\breve{\Omega}=\{(u,v)\in \mathbb{R}^2,~u\geq0,~v\geq0\}$. Additionally, we set $C:=c/(Kn)$, $S:=s/(Kr)$, $Q:=nq/r$ and $M:=m/K$, so $(M,S,Q,C)\in\Pi=(-1,1)\times \mathbb{R}^3_+$. By substitution of these new parameters into \eqref{hta1} we obtain
\begin{equation}\label{hta2} 
\begin{aligned}
\dfrac{du}{d\tau} & =  \ u\left(\left(1-u\right)\left(u-M\right) -Qv\right) \,,\\
\dfrac{dv}{d\tau} & =  \ \dfrac{Sv}{u+C}\left(u-v+C\right) \,,
\end{aligned}
\end{equation}

Note that system \eqref{hta1} is topologically equivalent to system \eqref{hta2} in $\Omega$ and the function $\varphi$ is a diffeomorphism preserving the orientation of time since $\det (\varphi(u,v,\tau))=Kn/r>0$ \cite{chicone}.

So, instead of analysing system \eqref{hta1} we analyse the topologically equivalent system \eqref{hta2}. Moreover, as $du/d\tau=uR(u,v)$ and $dv/d\tau=vW(u,v)$ with $R(u,v)=(1-u)(u-M)-Qv$ and $W(u,v)=S(u-v+C)/(u+C)$, system \eqref{hta2} is of Kolmogorov type. That is, the axes $u=0$ and $v=0$ are invariant. The $u$-nullclines of system \eqref{hta2} are $u=0$ and $v=(1-u)(u-M)/Q$, while the $v$-nullclines are $v=0$ and $v=u+C$. Hence, the equilibrium points for system \eqref{hta2} are  $(0,0)$, $(1,0)$, $(M,0)$, $(0,C)$ and the point(s) $(u^*,v^*)$ with $v^*=u^*+C$ and where $u^*$ is determined by the solution(s) of
\begin{align}\label{htae1}
\begin{aligned}
& (1-u)(u-M)/Q=u+C \,, \quad\text{or equivalently}\,,\\
& d(u)=u^2-(1+M-Q)u+(M+CQ)=0\,.
\end{aligned}
\end{align}
Define the functions $g(u)=(1-u)(u-M)/Q$ and $h(u)=(u+C)$ and observe that $\lim\limits_{u \rightarrow \pm \infty} g(u)= - \infty$ and $g(0)=-M$. So, \eqref{htae1} can have at most two positive real root, which are depending on the value of $M$ and $Q$, see Figure \ref{Fig.eqpoint}. Additionally, the solution of the equation \eqref{htae1} are given by
\begin{equation}\label{delta}
u_{1,2} = \dfrac12 \left(1+M-Q \pm \sqrt{\Delta} \right)\,, \quad {\rm with}
\,\, \Delta=(1+M-Q)^2-4(M+CQ)\,,
\end{equation}
such that $u_1\leq E \leq u_2<1$, where $E=(1+M-Q)/2$.

\begin{figure}
\centering
\includegraphics[width=14cm]{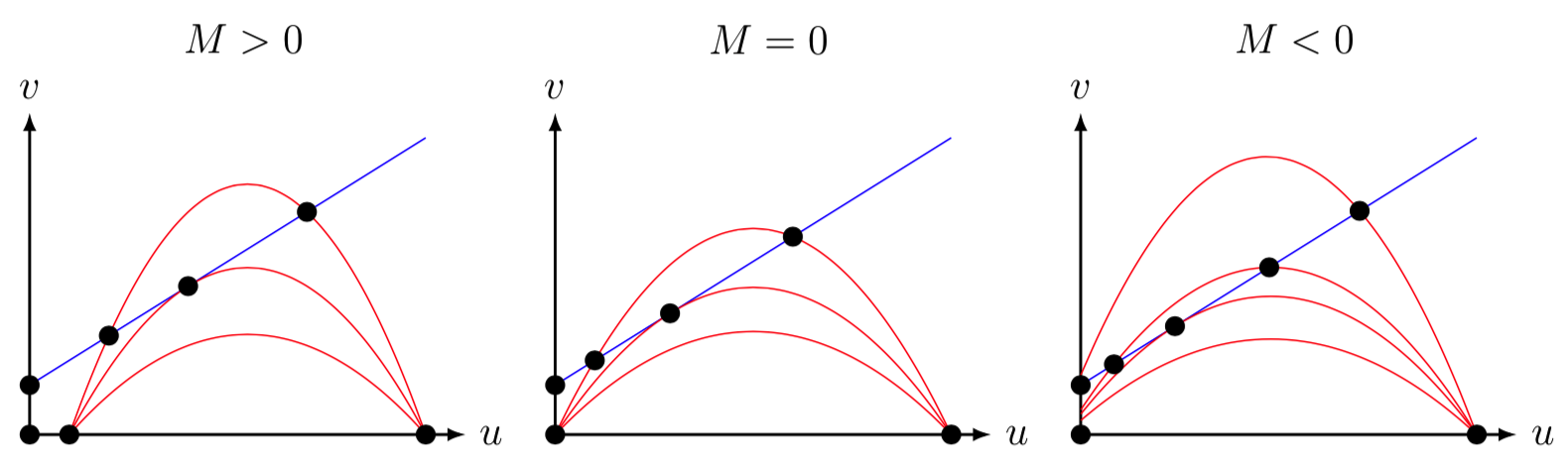}
\vspace*{-0.35cm} 
\caption{The intersections of the function $g(u)$ (red line) and the straight line $h(u)$ (blue lines) by changing $Q$ for the three possible cases of strong and weak Allee effect.}
\label{Fig.eqpoint}
\end{figure}

\subsection{Number of positive equilibrium points}\label{num}
Modifying the parameter $M$ and $Q$ impacts $\Delta$ and hence the number of positive equilibrium points. In particular, 
\begin{enumerate}
\item Strong Allee effect ($M>0$)
\begin{enumerate}[(a)]
\item If $1+M-Q>0$ and
\begin{enumerate}[(i)]
\item $\Delta<0$, then \eqref{hta2} has no equilibrium points in the first quadrant;
\item $\Delta>0$, then \eqref{hta2} has two equilibrium points $P_{1,2}=(u_{1,2},u_{1,2}+C)$ in the first quadrant; and
\item $\Delta=0$, then \eqref{hta2} has one equilibrium point $(E,E+C)$ of order two in the first quadrant, 
\end{enumerate}
\item If $1+M-Q\leq0$, then \eqref{hta2} has no equilibrium points in the first quadrant.
\end{enumerate}
\item Weak Allee effect ($M\leq0$)
\begin{enumerate}[(a)]
\item If $1+M-Q>0$ and $M+CQ>0$ and
\begin{enumerate}[(i)]
\item $\Delta<0$, then \eqref{hta2} has no equilibrium points in the first quadrant;
\item $\Delta>0$, then \eqref{hta2} has two equilibrium points $P_{1,2}=(u_{1,2},u_{1,2}+C)$ in the first quadrant; and
\item $\Delta=0$, then \eqref{hta2} has one equilibrium point $(E,E+C)$ of order two in the first quadrant, 
\end{enumerate}
\item If $1+M-Q>0$ and $M+CQ<0$ or $1+M-Q<0$ and $M+CQ<0$, then \eqref{hta2} has one equilibrium point $P_2$ in the first quadrant, since $u_1<0<u_2$.
\item If $1+M-Q>0$ and $M+CQ=0$, then $\Delta=(1+M-Q)^2$. Therefore, \eqref{hta2} has one equilibrium point $P_3=(u_{3},u_{3}+C)$ in the first quadrant, since $u_1=0$ and $u_2$ became $u_3=1+M-Q$.
\item If $1+M-Q=0$ and $M+CQ<0$, then $\Delta=-4(M+CQ)$. Therefore, \eqref{hta2} has one equilibrium point $P_4=(u_{4},u_{4}+C)$ in the first quadrant, since $u_1<0$ and $u_2$ became $u_4=\sqrt{-(M+CQ)}$.
\item If $1+M-Q\leq0$ and $M+CQ\geq0$, then \eqref{hta2} has no equilibrium points in the first quadrant.
\end{enumerate}
\end{enumerate}
\begin{rem}
Observe that none of these equilibrium points explicitly depend on the system parameter $S$. Therefore, $S$ is one of the natural candidates to act as bifurcation parameter.
\end{rem}

\section{Main Results}\label{res}

In this section, we discuss the stability of the equilibrium points of system \eqref{hta2} for strong and weak Allee effect. 
\begin{theorem} \label{the2}
All solutions of \eqref{hta2} which are initiated in the first quadrant are bounded and eventually end up in $\Phi=\{(u,v),\ 0\leq u\leq1,\ 0\leq v\leq1\}$.
\end{theorem}
\begin{proof}
First, observe that all the equilibrium points lie inside of $\Phi$. Additionally, as the system is of Kolmogorov type, the $u$-axis and  $v$-axis are invariant sets of \eqref{hta2}. Moreover, the set $\Gamma=\{(u,v),\ 0\leq u\leq1,\  v\geq0\}$ is an invariant region since $du/d\tau \leq0$ for $u=1$ and $v\geq0$. That is, trajectories entering into $\Gamma$ remain in $\Gamma$. On the other hand, by using the Poincar\'e compactification \cite{chicone, dumortier} which is given by the transformation $U=u/v$ and $V=1/v$ then $dU/d\tau=(1/v)du/d\tau-(u/v^2)dv/d\tau$ and $dV/d\tau=-(1/v^2)dv/d\tau$. Then, by applying the blowing-up method used in \cite{arancibia}, the result follows. 
\end{proof}

\subsection{The nature of the equilibrium points}

To determine the nature of the equilibrium points we compute the Jacobian matrix $J(u,v)$ of \eqref{hta2}
\[J(u,v)=\begin{pmatrix}\label{mat}
(1-u)(u-M)-Qv+u(1+M-2u)  & -Qu \\ 
\dfrac{Sv^2}{(u+C)^2}  &  \dfrac{S(u+C-2v)}{(u+C)} 
\end{pmatrix},\]
The stability of the equilibrium points $(0,0)$; $(1,0)$; $(M,0)$; and $(0,C)$ is 
\begin{lemma}\label{eqax}
The equilibrium points $(0,0)$ and $(1,0)$ are a saddle points and $(M,0)$ is unstable if $0<M<1$. Moreover, the equilibrium point $(0,C)$ is stable if $-CQ\leq M<1$ and a saddle point if $M<-CQ$. Furthermore, if there are no positive equilibrium points in the first quadrant, i.e. for $\Delta<0$ \eqref{delta}, then $(0,C)$ is globally asymptotically stable in the first quadrant.
\end{lemma}
\begin{proof}
The determinant and the trace of the Jacobian matrix \eqref{mat} evaluated at $(0,0)$ are $\det(J(0,0))=-MS$ and $\tr(J(0,0))=S-M$. Similarly, the determinant of the Jacobian matrix \eqref{mat} evaluated at $(1,0)$ is $\det(J(1,0))=-(1-M)S<0$, since $M<1$. Then, the determinant and the trace of the Jacobian matrix \eqref{mat} evaluated at $(M,0)$ are $\det(J(M,0))=SM(1-M)>0$ and $\tr(J(M,0))=M(1-M)+S>0$, since $M<1$. Finally,  the determinant and the trace of the Jacobian matrix \eqref{mat} evaluated at $(0,C)$ are $\det(J(0,C))= S(M+QC)$ and $\tr(J(0,C))=-(M+CQ+S)$. It follows that $(1,0)$ is always a saddle point and $(M,0)$ is unstable if the system \eqref{hta2} is affected by strong Allee effect ($0<M<1$). Moreover, the equilibrium point $(0,0)$ is a saddle point if $0\leq M<1$ and unstable equilibrium point if $M<0$. Furthermore, the equilibrium point $(0,C)$ is stable if $-CQ\leq M<1$ and a saddle point if $M<-CQ$.
Finally, by Theorem~\ref{the2} we have that solutions starting in the first quadrant are bounded and eventually end up in the invariant region $\Gamma$. Moreover, the equilibrium point $(1,0)$ is a saddle point and, if $\Delta<0$ as defined in \eqref{delta}, there are no equilibrium points in the interior of the first quadrant. Thus, by the Poincar\'e--Bendixson Theorem the unique $\omega$-limit of all the trajectories is the equilibrium point $(0,C)$.
\end{proof}

Next, we consider the stability of the two positive equilibrium points $P_{1,2}$ of system \eqref{hta2} in the interior of $\Phi$. These equilibrium points lie on the curve $u=v+C$ such that $g(u)=u+C$ \eqref{htae1}. The Jacobian matrix of system \eqref{hta2} at these equilibrium points becomes
\begin{align*}
J(u,u+C)=\begin{pmatrix}
u(1+M-2u)  & -Qu \\ 
S  & -S 
\end{pmatrix}.
\end{align*}
Moreover, the determinant and the trace of $J(u,u+C)$ are given by
\begin{align*}
\det(J(u,u+C)) =&Su(-1-M+Q+2u),\\
\tr(J(u,u+C)) &= u(1+M-2u)-S.
\end{align*}

Thus, the sign of the determinant depends on the sign of $-1-M+Q+2u$, while the sign of the trace depends on the sign of $ u(1+M-2u)-S$. This gives the following results.

\begin{lemma}\label{p1}
Let the system parameters of \eqref{hta2} be such that $1+M-Q>0$ and $M+CQ>0$. Then, $\Delta>0$, as defined in \eqref{delta}, and therefore the equilibrium point $P_1$ is a saddle point.
\end{lemma}
\begin{proof}
Evaluating $-1-M+Q+2u$ at $u_1$ gives:\\
\[\begin{aligned}
-1-M+Q+2u_1&=-1-M+Q+(1+M-Q-\sqrt{\Delta})=-\sqrt{\Delta}<0.
\end{aligned}\]
Hence, $\det(J(P_1))<0$ and $P_1$ is thus a saddle point.
\end{proof}

Note that if $M+CQ=0$ then the equilibrium point $P_1$ collapses with $(0,C)$ and if $M+CQ<0$ then the equilibrium point $P_1$ crosses to the second or third quadrant.
\begin{lemma}\label{p2}
Let the system parameters be such that $1+M-Q>0$ and $M+CQ>0$ or $1+M-Q>0$ and $M+CQ<0$ or $1+M-Q<0$ and $M+CQ<0$. Then, $\Delta>0$ \eqref{delta} and therefore the equilibrium point $P_2$ is:
\begin{enumerate}[(i)]
\item a repeller if $0<S<S_1:=\dfrac{1}{2}(1+M-Q+\sqrt{\Delta})(Q+\sqrt{\Delta})$; and
\item an attractor if $S>S_1$,
\end{enumerate}
\end{lemma}
\begin{proof}
Evaluating $-1-M+Q+2u$ at $u_2$ gives:\\
\[\begin{aligned}
-1-M+Q+2u_2&=-1-M+Q+(1+M-Q+\sqrt{\Delta})=\sqrt{\Delta}>0.
\end{aligned}\]
Hence, $\det(J(P_2))>0$. Evaluating $u(1+M-2u)-S$ at $u=u_2$ gives
\[\begin{aligned}
u(1+M-2u)=& \dfrac{1}{2}(1+M-Q+\sqrt{\Delta})(1+M-(1+M-Q+\sqrt{\Delta}))=S_1,\\
\end{aligned}\]
Therefore, the sign of the trace, and thus the behaviour of $P_2$ depends on the parity of $u_2(1+M-2u_2)-S_1$, see Figure \ref{Fig.summSA}. 
\end{proof}  

\begin{rem}
Note that if $M+CQ=0$ then $\Delta=(1+M-Q)^2$ and therefore $u(1+M-2u)-S=-(1+M-Q)(1+M-2Q)-S$. Hence the sign of the trace depends on the parity of $(1+M-Q)(1+M-2Q)-S$
\end{rem}

If $\Delta=0$ \eqref{delta} the equilibrium points $P_1$ and $P_2$ collapse such that  $u_1=u_2=E=(1+M-Q)/2$. Therefore, system \eqref{hta2} has one equilibrium point of order two in the first quadrant given by $(E,E+C)$. Consequently, we have that $\Delta=0$ and therefore $C=((1+M-Q)^2-4M)/4Q$.
\begin{lemma} \label{p1p2}
Let the system parameters be such that $\Delta=0$ \eqref{delta}. Then, the equilibrium point $(E,E+C)$ is:
\begin{enumerate}[(i)]
\item a saddle-node attractor if $S<S_2:=\dfrac{1}{2}Q(1+M-Q)$; and
\item a saddle-node repeller if $S>S_2$\,.
\end{enumerate}
\end{lemma}
\begin{proof}
Evaluating $-1-M+Q+2u$ at $u=E$ gives    
\[\begin{aligned}
-1-M+Q+2u_2&=-1-M+Q+(1+M-Q)=0.
\end{aligned}\]
Hence, $\det(J(P_2))=0$. Evaluating $u(1+M-2u)-S$ at $u=E$ gives
\[\begin{aligned}
u(1+M-2u)=& \dfrac{1}{2}Q(1+M-Q)=S_2,\\
\end{aligned}\]
Therefore, the behaviour of the equilibrium point $(E,E+C)$ depends on the parity of $Q(1+M-Q)/2-S_2$. 
\end{proof}
See Figure \ref{Fig.summp1p2} for phase portraits related to both cases of Lemma~\ref{p1p2}.

\begin{lemma}\label{p3}
Let the system parameters be such that $1+M-Q>0$ and $M+CQ=0$. Then, the equilibrium point $P_3$ is:
\begin{enumerate}[(i)]
\item a repeller if $0<S<S_3:=-(1+M-Q)(1+M-2Q)$; and
\item an attractor if $S>S_3$,
\end{enumerate}
\end{lemma}
\begin{proof}
Evaluating $-1-M+Q+2u$ at $u_3=1+M-Q$ gives:\\
\[\begin{aligned}
-1-M+Q+2u&=1+M-Q>0.
\end{aligned}\]
Hence, $\det(J(P_3))>0$. Evaluating $u(1+M-2u)-S$ at $u_3=1+M-Q$ gives
\[\begin{aligned}
u(1+M-2u)=& -(1+M-Q)(1+M-2Q)=S_3,\\
\end{aligned}\]
Therefore, the sign of the trace, and thus the behaviour of $P_3$ depends on the parity of $u_3(1+M-2u_3)-S_3$, see Figure \ref{Fig.summp2}. 
\end{proof} 

\begin{lemma}\label{p4}
Let the system parameters be such that $1+M-Q=0$ and $M+CQ<0$. Then, the equilibrium point $P_4$ is:
\begin{enumerate}[(i)]
\item a repeller if $0<S<S_4:=\sqrt{-M-CQ}\left(Q-2\sqrt{-M-CQ}\right)$; and
\item an attractor if $S>S_4$,
\end{enumerate}
\end{lemma}
\begin{proof}
Evaluating $-1-M+Q+2u$ at $u_4=\sqrt{-M-CQ}$ gives:\\
\[\begin{aligned}
-1-M+Q+2u&=\sqrt{-M-CQ}>0.
\end{aligned}\]
Hence, $\det(J(P_4))>0$. Evaluating $u(1+M-2u)-S$ at $u_4=\sqrt{-M-CQ}$ gives
\[\begin{aligned}
u(1+M-2u)=& \sqrt{-M-CQ}(Q-2\sqrt{-M-CQ})=S_4,\\
\end{aligned}\]
Therefore, the sign of the trace, and thus the behaviour of $P_4$ depends on the parity of $u_4(1+M-2u_4)-S_4$, see Figure \ref{Fig.summp2}. 
\end{proof}  

\begin{lemma}\label{noeq}
Let the system parameters be such that $1+M-Q>0$, $M+CQ>0$ and $\Delta<0$ \eqref{delta} or $1+M-Q\leq0$ and $M+CQ\geq0$. Then, there are no positive equilibrium points in the first quadrant. Therefore, $(0,C)$ is globally asymptotically stable in the first quadrant.
\end{lemma}

\begin{proof}
Finally, by Theorem~\ref{the2} we have that solutions starting in the first quadrant are bounded and eventually end up in the invariant region $\Gamma$. Moreover, the equilibrium point $(1,0)$ is a saddle point and, if $\Delta<0$ \eqref{delta}, there are no equilibrium points in the interior of the first quadrant. Thus, by the Poincar\'e--Bendixson Theorem the unique $\omega$-limit of all the trajectories is the equilibrium point $(0,C)$, see Figure \ref{Fig.summSA} (a).
\end{proof}

\subsection{Bifurcation Analysis}
In this section, we discuss some of the possible bifurcation scenarios of system \eqref{hta2}. Observe that the stability of $(0,0)$, $(1,0)$, $(M,0)$ and $P_{1}$ do not change the stability. Additionally, none of the equilibrium points $P_2$, $P_3$, $P_4$ and $(E,E+C)$ explicitly depend on the system parameter $S$. Therefore, $S$ is one of the natural candidates to act as bifurcation parameter.

\begin{theorem}\label{the5}
Let the system parameters be such that $\Delta=0$ \eqref{delta}. Then, system \eqref{hta2} experiences a saddle-node bifurcation at the equilibrium point $(E,E+C)$ (for changing $Q$).
\end{theorem}
\begin{proof}
The proof of this theorem is based on Sotomayor's Theorem \cite{perko}.  
For $\Delta=0$, there is only one equilibrium point $(E,E+C)$ in the first quadrant, with $E=(1+M-Q)/2$. 
From the proof of Lemma~\ref{p1p2} we know that $\det(J(E,E+C))=0$ if $\Delta=0$. Additionally, let $U=(1,1)^T$ the eigenvector corresponding to the eigenvalue $\lambda=0$ of the Jacobian matrix $J(E,E+C)$, and let 
$$W=\left(\dfrac{4S}{(3+3M-5Q)(1+M-Q)},1\right)^T$$ be the eigenvector corresponding to the eigenvalue $\lambda=0$ of the transposed Jacobian matrix $J(E,E+C)^T$. 
	
If we represent \eqref{hta2} by its vector form
\begin{align} \nonumber %\label{eq:vf}
F(u,v;Q) =\begin{pmatrix}
(1-u)(u-M)-Qv\\ 
u-v+C
\end{pmatrix},
\end{align}
then differentiating $F$ at $(E,E+C)$
with respect to the bifurcation parameter $Q$ gives
\[F_Q(E,E+C;Q)=\begin{pmatrix}
-\dfrac{1}{2}(1+M-Q+2C)\\ 
0
\end{pmatrix}.\]
Therefore,
\[W \cdot F_Q(E,E+C;Q)=-\dfrac{2S(1+M-Q+2C)}{(3+3M-5Q)(1+M-Q)}\\ \neq0.\]
Next, we analyse the expression $W \cdot [D^2F(E,E+C;Q)(U,U)]$.
Therefore, we first compute the Hessian matrix $D^2F(u,v;Q)(V,V)$, where $V=(v_1,v_2)$,	
\[\begin{aligned}
D^2F(u,v;Q)(V,V) =& \dfrac{\partial^2F(u,v;Q)}{\partial u^2}v_1v_1+\dfrac{\partial^2F(u,v;Q)}{\partial u\partial v}v_1v_2+\dfrac{\partial^2F(u,v;Q)}{\partial v\partial u}v_2v_1\\
&+\dfrac{\partial^2F(u,v;Q)}{\partial v^2}v_2v_2\,.
\end{aligned}\]
At the equilibrium point $(E,E+C)$ and $V=U$, this simplifies to
\[\begin{aligned}	
D^2F(E,E+C;Q)(U,U)& = \begin{pmatrix}
2(-2+M-Q)\\ 
-\dfrac{2C^2S}{(1+C)^3}
\end{pmatrix}\,. 
\end{aligned}\]
Therefore
\[\begin{aligned}
W \cdot [D^2F(E,E+C;Q)(U,U)]=-\dfrac{8S(2-M+Q)}{(3+3M-5Q)(1+M-Q)}-\dfrac{2SC^2}{(1+C)^3}\neq0 \,.
\end{aligned}\]
By Sotomayor's Theorem \cite{perko} it now follows that system \eqref{hta2} has a saddle-node bifurcation at the equilibrium point $(E,E+C)$, see Figure \ref{Fig.bif} and Figure \ref{Fig.summp1p2}.
\end{proof}

\begin{theorem} \label{the4}
Let the system parameters be such that $\Delta=0$ \eqref{delta} and $S=Q(1+M-Q)/2$. Then, system \eqref{hta2} experiences a Bogdanov--Takens bifurcation at the equilibrium point $(E,E+C)$ (for changing $(Q,S)$).
\end{theorem}
\begin{proof}
If $\Delta=0$, or equivalently $Q=1+M-2E$, and $E(1+M-2E)=S$, then the Jacobian matrix of system \eqref{hta2} evaluated at the equilibrium point $(E,E+C)$ simplifies to
\[\begin{aligned}
J(E,E+C) &=\begin{pmatrix}
E(1+M-2E) & -E(1+M-2E) \\ 
E(1+M-2E) & -E(1+M-2E) 
\end{pmatrix},\\
&=\dfrac{1}{2}Q(M-Q+1) \begin{pmatrix}
1 & -1 \\ 
1 & -1 
\end{pmatrix}. 
\end{aligned}\]
So, $\det(J(E,E+C))=0$ and $tr(J(E,E+C))=0$.
Next, we find the Jordan normal form of $J(E,E+C)$. The latter has two zero eigenvalues with eigenvector $\psi^1=(1,1)^T$. This vector will be the first column of the matrix of transformations $\Upsilon$. For the second column of $\Upsilon$ we choose the generalised eigenvector $\psi^2=(1,0)^T$. Thus,
$\Upsilon = \begin{pmatrix} 
1 & 1 \\ 
1 & 0 
\end{pmatrix}$ and 
\[ \begin{aligned}
\Upsilon^{-1}(J(E,E+C))\Upsilon &= \dfrac{1}{2}Q(1+M-Q)\begin{pmatrix}
0 & 1 \\ 
0 & 0 
\end{pmatrix} . 
\end{aligned}\]
Hence, we have the Bogdanov--Takens bifurcation \cite{perko}, or bifurcation of codimension two, and the equilibrium point $(E,E+C)$ is a cusp point for $(Q,S)=(Q_2,S_2(Q_2))$ such that $\Delta=0$ and $E(1+M-2E)=S$ \cite{xiao2}, see Figure \ref{Fig.bif} and Figure \ref{Fig.summp1p2}.
\end{proof}

\begin{figure}
\centering
\includegraphics[width=10cm]{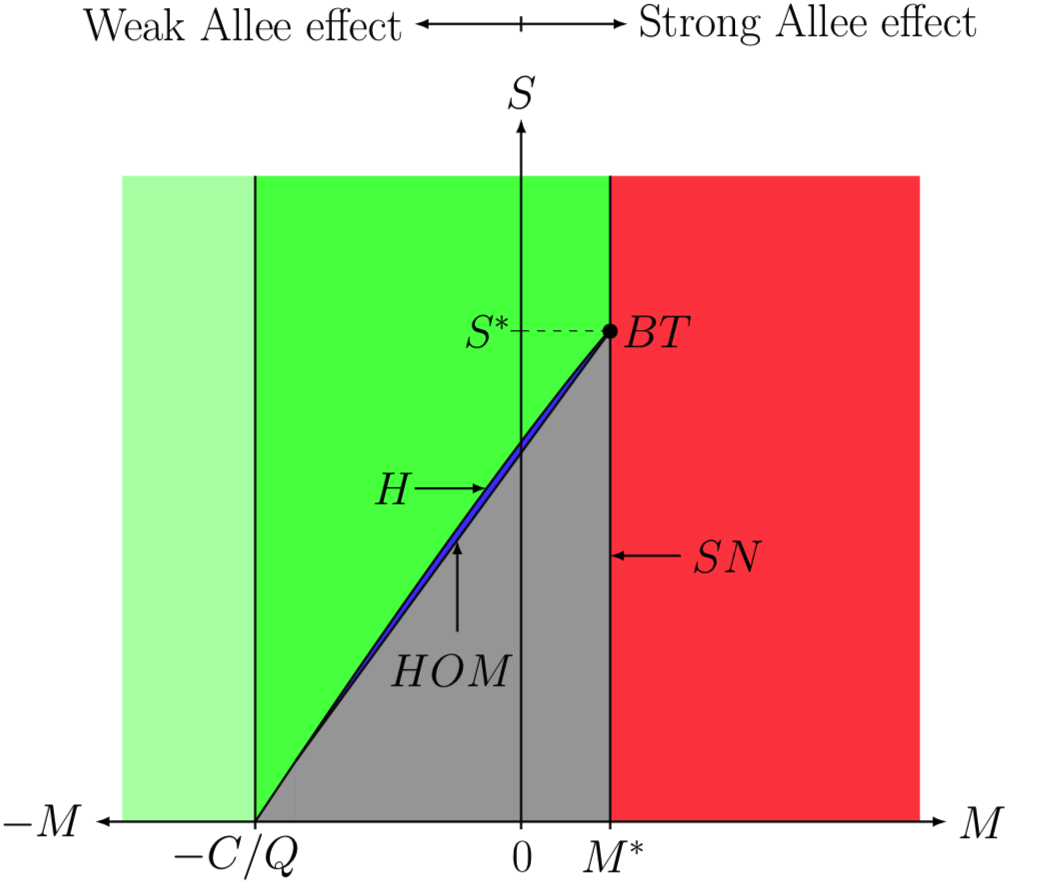}
\vspace*{-0.35cm}
\caption{The bifurcation diagram of system \eqref{hta2} with strong ($M>0$) and weak ($M\leq0$) Allee effect for $(Q,C)$ fixed and created with the numerical bifurcation package MATCONT \cite{matcont}. The curve H represents the Hopf curve where $P_2$ changes stability (Lemma~\ref{p2}), HOM represents the homoclinic curve of $P_1$, SN represents the saddle-node curve from Lemma~\ref{p1p2} where $\Delta=0$, and $BT$ represents the Bogdanov--Takens bifurcation from Theorem~\ref{the4} where $\Delta=0$.}
\label{Fig.bif}
\end{figure}
The bifurcation curves obtained from Lemma~\ref{p2}, Lemma~\ref{p1p2}, and  Theorem \ref{the5} divide the $(M,S)$-parameter-space into five parts, see Figure \ref{Fig.bif}. Modifying the parameter $M$ -- while keeping the other two parameters $(Q,C)$ fixed -- impacts the number of positive equilibrium points of system \eqref{hta2}. The modification of the parameter $S$ changes the stability of the positive equilibrium point $P_2$ of system \eqref{hta2}, while the other equilibrium points $(0,0)$, $(M,0)$, $(1,0)$ and $P_1$ do not change their behaviour. 
There are no positive equilibrium points in system \eqref{hta2} when the parameters $M,S$ are located in the red area where $\Delta<0$ \eqref{delta}. In this case, the equilibrium point $(0,C)$ is a global attractor, see Lemma~\ref{noeq} and Figure \ref{Fig.summSA} (a). For $M=M^*$, which is the saddle-node curve SN in Figure~\ref{Fig.bif}, the equilibrium points $P_1$ and $P_2$ collapse since $\Delta=0$, see Lemma \ref{p1p2} and Figures \ref{Fig.summp1p2} (a) and (b). So, system \eqref{hta2} experiences a saddle-node bifurcation and a Bogdanov--Takens bifurcation (labeled BT in Figure~\ref{Fig.bif}) along this line, see Theorems \ref{the5} and \ref{the4}, and see also Figure \ref{Fig.summp1p2} (c). When the parameter $M$ is located in $-C/Q<M<M^*$, system \eqref{hta2} has two equilibrium points $P_1$ and $P_2$. The equilibrium point $P_1$ is always a saddle point, see Lemma~\ref{p1}, while $P_2$ can be unstable or stable. For $(M,S)$ in the grey area the equilibrium point $P_2$ is unstable, see Figure~\ref{Fig.summSA} (a). For $(M,S)$ in the blue area the stable equilibrium point $P_2$ is surrounded by a stable limit cycle, see also Figures~\ref{Fig.summSA} (b). For $(M,S)$ in the  green area the equilibrium point $P_2$ is stable, see Figure~\ref{Fig.summSA} (c). Finally, for $(M,S)$ in the light green area system \eqref{hta2} has only one equilibrium point in the first quadrant which is always stable, see Lemmas \ref{p3} and \ref{p4}. Since $P_1$ collapse with $(0,C)$ or crosses to the second or third quadrant, see Figure~\ref{Fig.summp2})
  
\section{Basins of Attraction}\label{bas}

For system parameters $(Q,M,A,C)$ such that the conditions 1.(a) (strong Allee effect) and 2.(a) (weak Allee effect) presented in Subsection \ref{num} are met and for $S>S_1$, system \eqref{hta2} has two attractors, namely $(0,C)$ and $P_2$. Furthermore, at the critical value $S=S_1$, such that $tr(J(P_2))=0$, $P_2$ undergoes a Hopf bifurcation \cite{chicone}. Note that $S_1$ depends on $Q$ and it can actually be negative. In that case $P_2$ is an attractor for all $S>0$ (and as long as $\Delta>0$).

Next, we discuss the basins of attraction of the attractors $(0,C)$ and $P_2$ (for $S>S_1$) in $\Phi$ (see Theorem~\ref{the2}). The stable manifold of the saddle point $P_1$, $W^s(P_1)$, often acts as a separatrix curve between these two basins of attraction. 

Let $W^{u,s}_{\nearrow}(P_1)$ be the (un)stable manifold of $P_1$ that goes up to the right (from $P_1$) and let $W^{u,s}_{\swarrow}(P_1)$ be the (un)stable manifold of $P_1$ that goes down to the left (from $P_1$). From the phase plane and the nullclines of system \eqref{hta2} it immediately follows that $W^s_{\nearrow}(P_1)$ is connected with $(M,0)$ and $W^u_{\swarrow}(P_1)$ with $(0,C)$. Furthermore, everything in between of $W^{s}_{\nearrow}(P_1)$, $W^{u}_{\swarrow}(P_1)$ and the $u$-axis also asymptotes to the origin. 

For $\Delta>0$, $M>-CQ$ and depending on the value of $S$, there are different cases for the boundary of the basins of attraction in the invariant region $\Phi$, see Theorem~\ref{the2}. By continuity of the vector field in $S$, see \eqref{hta2}, we get:
\begin{enumerate}[(i)]
\item For $S$ on the grey region in the bifurcation diagram showed in Figure \ref{Fig.bif}, the equilibrium point $P_2$ is unstable, see lemma \ref{p2}, and $W^u_{\nearrow}(P_1)$ connects with $(0,C)$. Hence, $\Phi$ is the basin of attraction of $(0,C)$, see Figure \ref{Fig.summSA} (a) and (d).
\item For $S$ on the blue region in the bifurcation diagram showed in Figure \ref{Fig.bif}. There is a stable limit cycle that surrounds $P_2$ and $W^u_{\nearrow}(P_1)$ connects with this limit cycle. This limit cycle is created around $P_2$ via the Hopf bifurcation \cite{gaiko}. Therefore, $W^s(P_1)$ forms a separatrix curve between the basins of attraction of $P_2$ and $(0,C)$ in this parameter regime, see Figure \ref{Fig.summSA} (b) and (e).
\item  For $S$ on the green region in the bifurcation diagram showed in Figure \ref{Fig.bif}. Then, $W^s_{\swarrow}(P_1)$ intersects the boundary of $\Phi$, and $W^s(P_1)$ again forms the separatrix curve in $\Phi$, see Figure \ref{Fig.summSA} (c) and (f).
\end{enumerate}

\begin{figure}
\centering
\includegraphics[width=13cm]{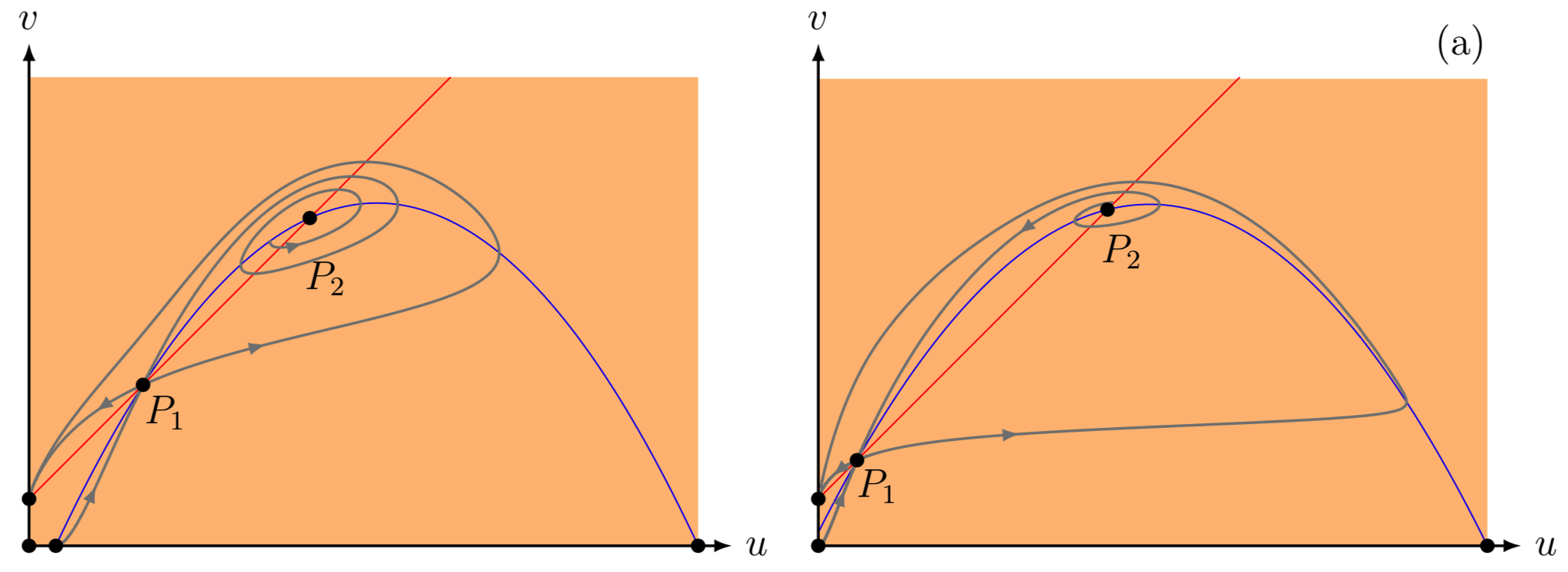}
\includegraphics[width=13cm]{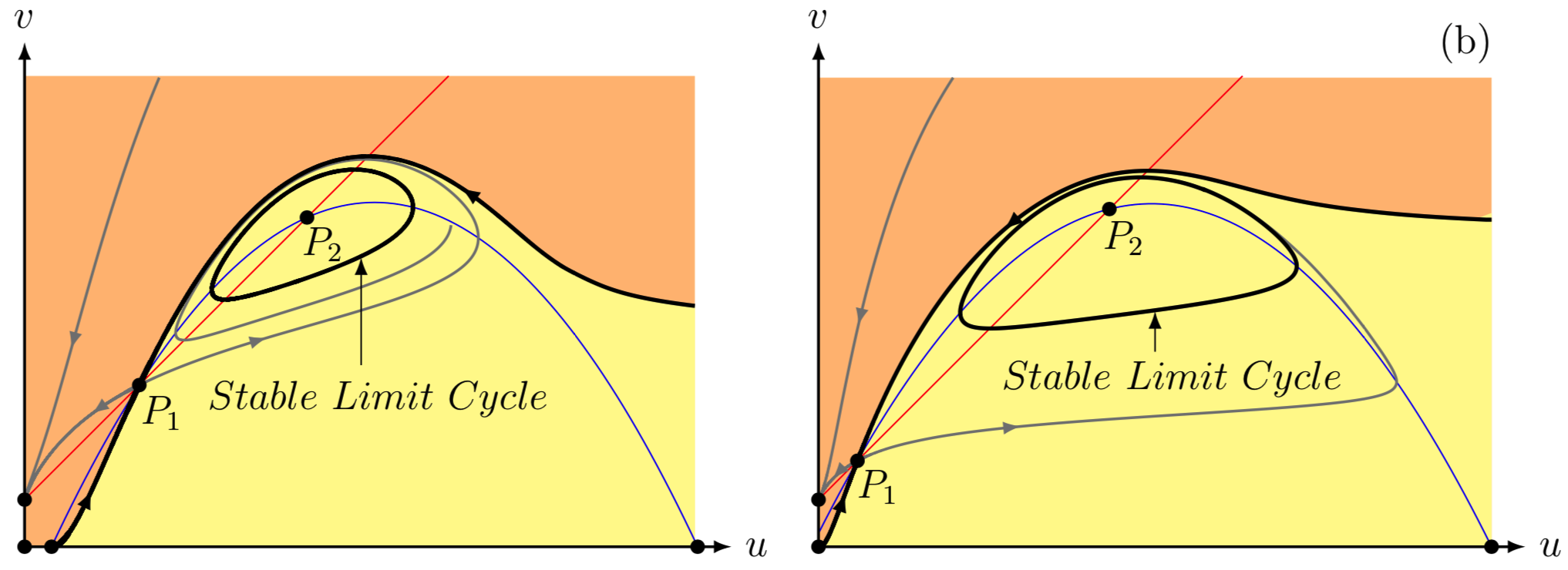}
\includegraphics[width=13cm]{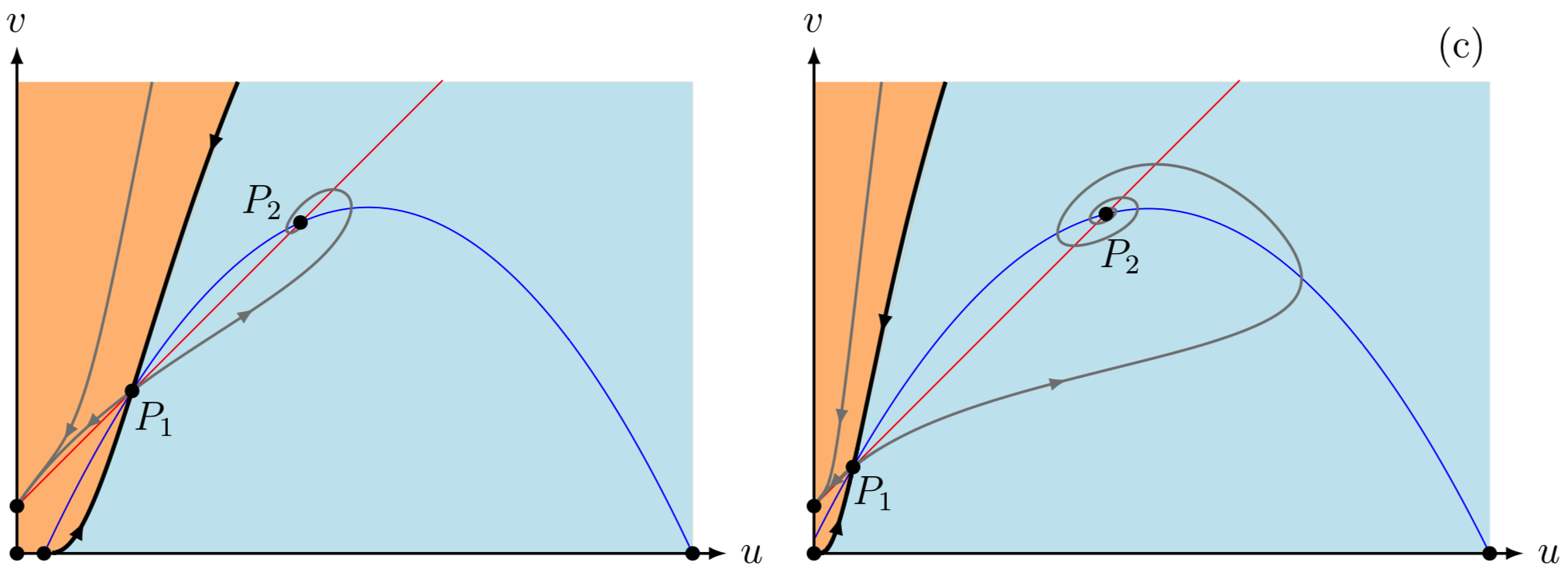}
\vspace*{-0.35cm} 
\caption{The grey region represent the basin of attraction of $P_2$, the yellow region represent the basin of attraction of a stable limit cycle and the orange region represent the basin of attraction of $(0,C)$. Moreover, the blue (red) curve represents the prey (predator) nullcline. For $C=0.07$, and $Q=0.45$, such that $\Delta>0$ \eqref{delta} and (a) $M=0.04$ (left panel) and $M=-0.01$ (right panel) the equilibrium point $(0,C)$ is global attractor. (b) $M=0.04$ (left panel) and $M=-0.01$ (right panel) the equilibrium point $P_2$ is surrounded by a stable limit cycle (grey curve) and $W^s(P_1)$ forms the separatrix curve in $\Phi$. (c) $M=0.04$ (left panel) and $M=-0.01$ (right panel) the equilibrium points $(0,C)$ and $P_2$ are attractors and $W^s(P_1)$ forms the separatrix curve in $\Phi$. Observe that the same color conventions are used in the upcoming figures. Moreover, Tikz and Matlab were used to do the simulations.}
\label{Fig.summSA}
\end{figure}

Note that the system parameters $(Q,C)$ are fixed at $(0.45, 0.07)$ and $M=0.04$ in the left panel of Figures~\ref{Fig.summSA} (a)-(c). Consequently, $u_{1,2}$ are constant. In particular, $u_1 \approx 0.1704$ and $u_2 \approx 0.4196$. Similarly, for $(Q,C)$ fixed at $(0.45, 0.07)$ and $M=-0.01$ in the right panel of Figures~\ref{Fig.summSA} (a)-(c). Consequently, $u_{1,2}$ are also constant. In particular, $u_1 \approx 0.05785$ and $u_2 \approx 0.43215$.

For $\Delta=0$, $M>-CQ$ and depending on the value of $S$, there are three different cases for the boundary of the basins of attraction in the invariant region $\Phi$, see Theorem~\ref{the2}. By continuity of the vector field in $S$, see \eqref{hta2}, we get:
\begin{enumerate}[(i)]
\item For $0<S<S_2$, the equilibrium point $(E,E+C)$ is a saddle-node attractor, see Lemma \ref{p1p2}, see Figure \ref{Fig.summp1p2} (a).
\item For $S_2<S$, the equilibrium point $(E,E+C)$ is a saddle-node repeller, see Lemma \ref{p1p2}, see Figure \ref{Fig.summp1p2} (b).
\item For $S=S_2$, the equilibrium point $(E,E+C)$ is a cusp point, see Theorem \ref{the4}, see Figure \ref{Fig.summp1p2} (c).
\end{enumerate}

\begin{figure}
\centering
\includegraphics[width=13cm]{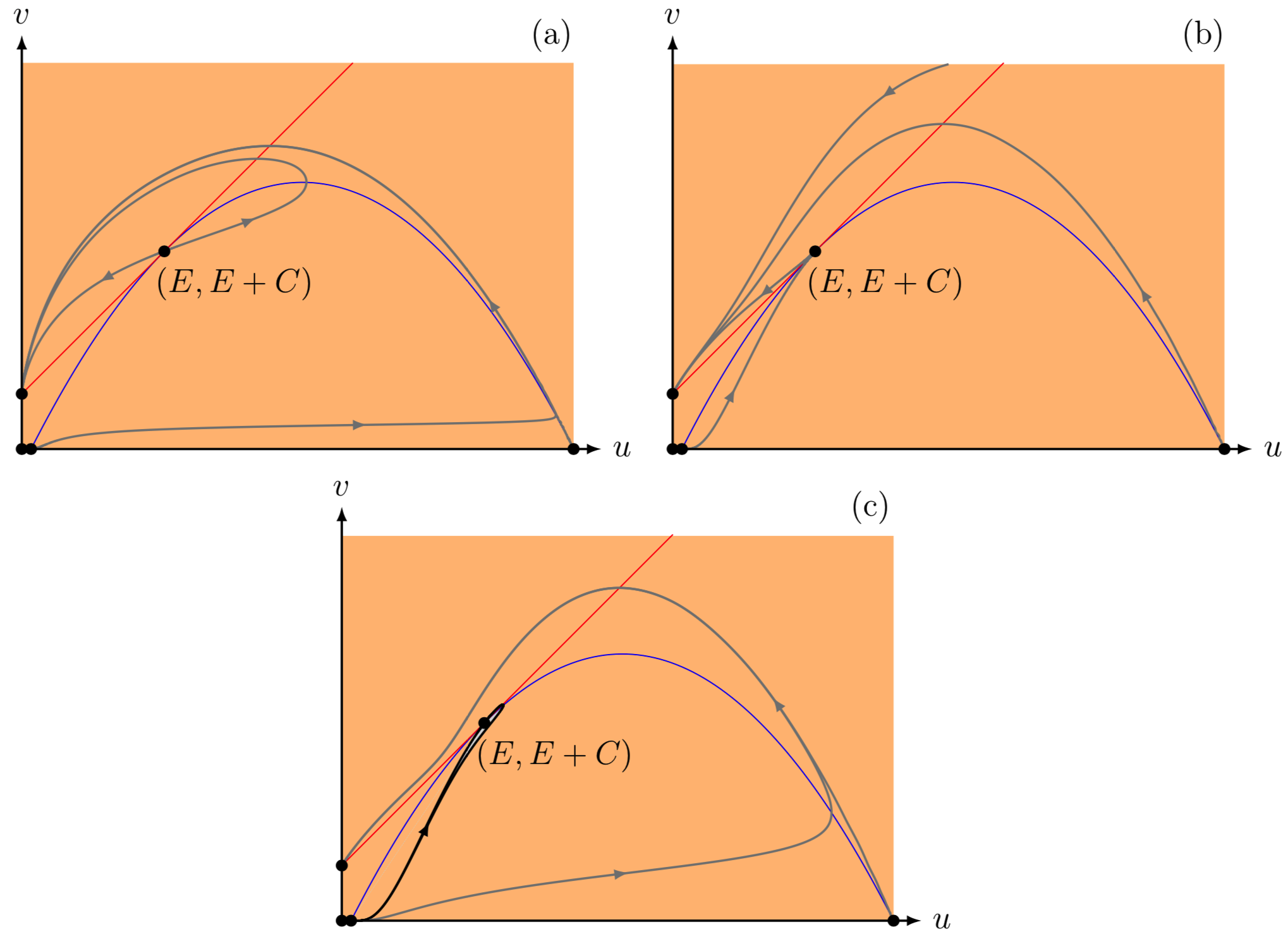}
\vspace*{-0.35cm} 
\caption{For $C=0.1$, $Q=0.5$, and $M=0.01676030$ such that $\Delta=0$ \eqref{delta}, system \eqref{hta2} has one equilibrium point $(E,E+C)=(0.25833,0.35833)$ of order two. (a) For $S<S_2=0.12919012$, the equilibrium point $(E,E)$ is a saddle-node repeller. (b) For $S>S_2$, the equilibrium point $(E,E+C)$ is a saddle-node attractor. (c) For $S=S_2$, the equilibrium point $(E,E+C)$ is a cusp point (See Figure \ref{Fig.summSA} for the color conventions.)}
\label{Fig.summp1p2}
\end{figure}

For system parameters $(Q,M,A,C)$ such that the conditions 2.(b), 2.(c) and 2.(d) presented in Subsection \ref{num} are met (weak Allee effect), system \eqref{hta2} has one attractor in the first quadrant, namely $P_{3,4}$.

For $1+M-Q>0$ and $M+CQ\leq0$ or $1+M-Q<0$ and $M+CQ\leq0$ or $1+M-Q=0$ and $M+CQ<0$ and depending on the value of $S$, there are three different cases for the boundary of the basins of attraction in the invariant region $\Phi$, see Theorem~\ref{the2}. By continuity of the vector field in $S$, see \eqref{hta2}, we get:
\begin{enumerate}[(i)]
\item For $S$ on the grey region in the bifurcation diagram showed in Figure \ref{Fig.bif}, the equilibrium point $P_{3,4}$ is unstable (see lemma \ref{p3} and \ref{p4}). Hence, $\Phi$ is the basin of attraction of $(0,C)$, see Figure \ref{Fig.summp2} (a).
\item For $S$ on the blue region in the bifurcation diagram showed in Figure \ref{Fig.bif}, the equilibrium point $P_{3,4}$ is unstable surrounded by a stable limit cycle (see lemma \ref{p3} and \ref{p4}). Therefore, $\Phi$ is the basin of attraction of the stable limit cycle, see Figure \ref{Fig.summp2} (b).
\item  For $S$ on the green region in the bifurcation diagram showed in Figure \ref{Fig.bif}, the equilibrium point $P_{3,4}$ is stable (see lemma \ref{p3} and \ref{p4}). Hence, $\Phi$ is the basin of attraction of the equilibrium point $P_{3,4}$, see Figure \ref{Fig.summp2} (c).
\end{enumerate}

Note that the system parameters $(Q,C)$ are fixed at $(0.55, 0.1)$ and $M=-0.055$ in Figures~\ref{Fig.summp2} (a)-(c). Consequently, $u_{3}$ is constant. In particular, $u_3 \approx 0.395$. Similarly, for $(Q,C)$ fixed at $(0.55, 0.1)$ and $M=-0.1$ in Figure~\ref{Fig.summSA} (d). Consequently, $u_{4}$ are also constant. In particular, $u_4 \approx 0.45$.

\begin{figure}
\centering
\includegraphics[width=13cm]{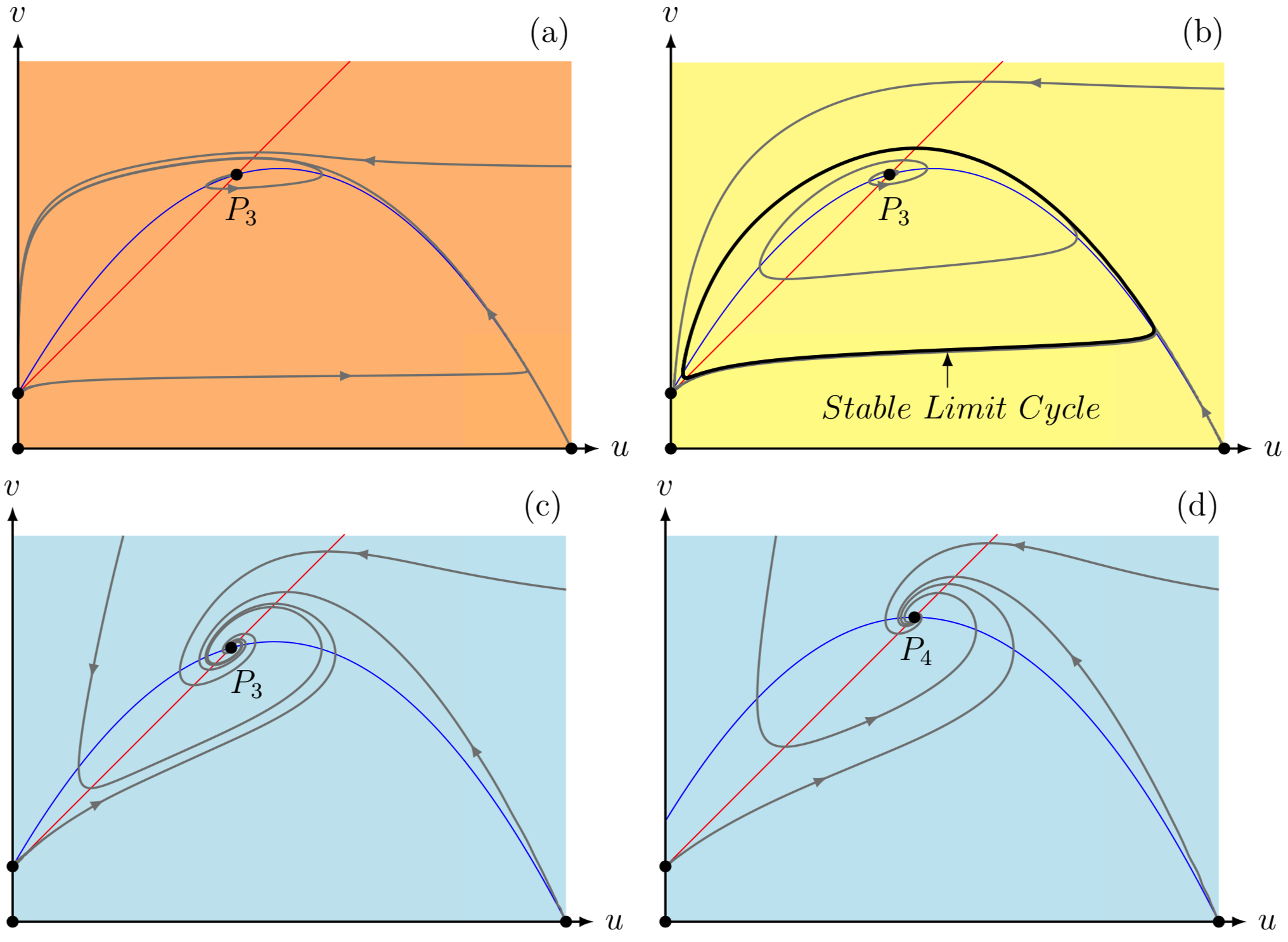}
\vspace*{-0.35cm} 
\caption{For $M=-0.055$, $C=0.1$, and $Q=0.55$, such that $u_1=0$ and (a) $S=0.01<S_k$, the equilibrium point $(0,C)$ is global attractor. (b) For $S_k<S=0.03<S_k^*$ the equilibrium point $P_3$ is surrounded by a stable limit cycle (black curve). (c) For $S=0.15>S_k^*$ the equilibrium point $P_3$ is stable. Similarly, for $M=-0.1$, $C=0.1$, and $Q=0.55$, such that $u_1<0$ and (d) For $S=0.19$ the equilibrium point $P_4$ is stable (See Figure \ref{Fig.summSA} for the color conventions).}
\label{Fig.summp2}
\end{figure}

\section{Conclusions}
In this manuscript, the Holling--Tanner predator-prey model with strong and weak Allee effect and functional response Holling type I was studied. Using a diffeomorphism we analysed a topologically equivalent system \eqref{hta2}. This system has four system parameters which determine the number and the stability of the equilibrium points. We showed that the equilibrium points $(1,0)$ and $P_1$ are always saddle points, $(M,0)$ is an unstable point. Moreover,  the equilibrium point $(0,0)$ can be a saddle or unstable equilibrium point and the equilibrium point can be stable or saddle point, see Lemmas \ref{eqax} and \ref{p1}. In contrast, the equilibrium point $P_2$ can be an attractor or a repeller, depending on the trace of its Jacobian matrix, see Lemma \ref{p2}. Furthermore, for some sets of parameters values the equilibrium point $P_1$ can collapses with $(0,C)$ and then crosses tho the second or third quadrant. Therefore, there exist one positive equilibrium point ($P_i$ with $i=3,4$) which can be an attractor or a repeller, depending on the trace of its Jacobian matrix, see Lemmas \ref{p3} and \ref{p4}. Additionally, the stable manifold of the equilibrium point $P_1$ determines a separatrix curve which divides the basins of attraction of $(0,C)$ and $P_2$, see Figure \ref{Fig.summSA}.

The equilibrium points $P_1$ and $P_2$ collapse for $\Delta=0$ \eqref{delta} and system \eqref{hta2} experiences a saddle-node bifurcation, see Theorem \ref{the5}. Additionally, for $S=f(E)$ we obtain a cusp point (Bogdanov--Takens bifurcation) \cite{xiao2}, see Theorem \ref{the4}. We summarise the behavior for changing parameters $S$ and $Q$ in Figure \ref{Fig.summp1p2}.

Additionally, we showed that the Allee effect (strong and weak) in the Holling--Tanner model \eqref{hta1} modified the dynamics of the original Holling--Tanner model \eqref{ht1}. Gonzalez-Olivares {\em et al.} \cite{gonzalez} showed that system \eqref{mht1} with $c=0$ has the extinction of both population and/or coexistence.     

Since the function $\varphi$ is a diffeomorphism preserving the orientation of time, the dynamics of system \eqref{hta2} is topologically equivalent to system \eqref{hta1}. Therefore, we can conclude that for certain population sizes, there exists self-regulation in system \eqref{hta1}, that is, the species can coexist. Moreover, for some sets of parameters values system \eqref{hta1} experiences an oscillations of the populations.
However, system \eqref{hta1} is sensitive to disturbances of the parameters, see the changes of the basin of attraction of $P_2$, $P_3$ and $P_4$ in Figures \ref{Fig.summSA} and \ref{Fig.summp2}. In addition, we showed that the self-regulation depends on the values of the parameters $S$ and $M$. Since $S:=s/(Kr)$, this, for instance, implies that increasing the intrinsic growth rate of the predator $r$ -- or the carrying capacity $K$ -- decreases the area of coexistence (related to basins of attraction of $P_2$, $P_3$ or $P_4$ in \eqref{hta2}), or, equivalent, decreasing the intrinsic growth rate of the prey $s$ decreases this area of the coexistence. Similar statements can be derived for the other system parameters of \eqref{hta1}. The impact on the basin of attraction by changing the intrinsic growth rate of the predator and the Allee threshold population level is showed in Figures \ref{Fig.summSA} and \ref{Fig.summp2}. We can see that strong Allee effect reduces the basin of attraction of the positive equilibrium point. Therefore, it reduces the coexistence and/or oscillation of both populations.

Additionally, we showed that the strong Allee effect in the Holling--Tanner model \eqref{hta2} does not modified the dynamics of system \eqref{hta2} affected by weak Allee effect (i.e. $-CQ<M\leq0$). We also proved that system \eqref{hta2} with $-CQ<M<M^*$ has always two positive equilibrium points $P_1$ and $P_2$, see Figures \ref{Fig.eqpoint} and \ref{Fig.bif}. 

\section*{References}
\bibliography{References.bib}

\begin{thebibliography}{10}
\expandafter\ifx\csname url\endcsname\relax
  \def\url#1{\texttt{#1}}\fi
\expandafter\ifx\csname urlprefix\endcsname\relax\def\urlprefix{URL }\fi
\expandafter\ifx\csname href\endcsname\relax
  \def\href#1#2{#2} \def\path#1{#1}\fi

\bibitem{yu}
S.~Yu, Global asymptotic stability of a predator-prey model with modified
  {Leslie--Gower} and {Holling--Type} {II} schemes, Discrete Dynamics in Nature
  and Society 2012.

\bibitem{zhao}
Z.~Zhao, L.~Yang, L.~Chen, Impulsive perturbations of a predator--prey system
  with modified {Leslie--Gower} and {Holling} type {II} schemes, Journal of
  Applied Mathematics and Computing 35 (2011) 119--134.

\bibitem{arrows}
D.~Arrowsmith, C.~Chapman, Dynamical systems: {Differential} equations, maps
  and chaotic behaviour, Computers and Mathematics with Applications 32 (1996)
  132--132.

\bibitem{saez}
E.~Saez, E.~Gonzalez-Olivares, Dynamics on a predator--prey model, SIAM Journal
  on Applied Mathematics 59 (1999) 1867--1878.

\bibitem{banerjee1}
M.~Banerjee, Turing and non-{Turing} patterns in two-dimensional prey-predator
  models, Applications of Chaos and Nonlinear Dynamics in Science and
  Engineering 4 (2015) 257--280.

\bibitem{ghazaryan}
A.~Ghazaryan, V.~Manukian, S.~Schecter, Travelling waves in the
  {Holling--Tanner} model with weak diffusion, Proceedings of the Royal Society
  A: Mathematical, Physical and Engineering Sciences 471 (2015) 20150045.

\bibitem{martinez}
N.~Mart{\'\i}nez-Jeraldo, P.~Aguirre, Allee effect acting on the prey species
  in a {Leslie--Gower} predation model, {Nonlinear Analysis: Real World
  Applications} 45 (2019) 895--917.

\bibitem{hanski}
I.~Hanski, H.~Henttonen, E.~Korpim{\"a}ki, L.~Oksanen, P.~Turchin,
  Small--rodent dynamics and predation, Ecology 82 (2001) 1505--1520.

\bibitem{hanski2}
I.~Hanski, L.~Hansson, H.~Henttonen, Specialist predators, generalist
  predators, and the microtine rodent cycle, The Journal of Animal Ecology
  (1991) 353--367.

\bibitem{hanski3}
I.~Hanski, P.~Turchin, E.~Korpimaki, H.~Henttonen, Population oscillations of
  boreal rodents: regulation by mustelid predators leads to chaos, Nature 364
  (1993) 232.

\bibitem{turchin2}
P.~Turchin, I.~Hanski, An empirically based model for latitudinal gradient in
  vole population dynamics, The American Naturalist 149 (1997) 842--874.

\bibitem{turchin}
P.~Turchin, Complex population dynamics: a theoretical/empirical synthesis,
  Vol.~35 of Monographs in population biology, {Princeton} {University}
  {Press}, Princeton, N.J., 2003.

\bibitem{may}
R.~May, Stability and complexity in model ecosystems, Vol.~6 of Monographs in
  population biology, {Princeton} {University} {Press}, Princeton, N.J., 1974.

\bibitem{freedman}
H.~Freedman, Deterministic mathematical models in population ecology, Pure and
  applied mathematics (Dekker); 57, Wiley, New York, 1980.

\bibitem{aziz}
M.~Aziz-Alaoui, M.~Daher, Boundedness and global stability for a predator--prey
  model with modified {Leslie--Gower} and {Holling}--type {II} schemes, Applied
  Mathematics Letters 16 (2003) 1069--1075.

\bibitem{arancibia}
C.~Arancibia-Ibarra, E.~Gonzalez-Olivares, A modified {Leslie--Gower}
  predator--prey model with hyperbolic functional response and {Allee} effect
  on prey, BIOMAT 2010 International Symposium on Mathematical and
  Computational Biology (2011) 146--162.

\bibitem{feng}
P.~Feng, Y.~Kang, Dynamics of a modified {Leslie-Gower} model with double
  {Allee} effects, Nonlinear Dynamics 80 (2015) 1051--1062.

\bibitem{singh}
A.~Singh, S.~Gakkhar, Stabilization of modified {Leslie--Gower} prey--predator
  model, Differential Equations and Dynamical Systems 22 (2014) 239--249.

\bibitem{korobei}
A.~Korobeinikov, A {Lyapunov} function for {Leslie--Gower} predator--prey
  models, Applied Mathematics Letters 14 (2001) 697--699.

\bibitem{kramer}
A.~Kramer, L.~Berec, J.~Drake, Allee effects in ecology and evolution, Journal
  of Animal Ecology 87 (2018) 7--10.

\bibitem{allee}
W.~Allee, O.~Park, A.~Emerson, T.~Park, K.~Schmidt, Principles of animal
  ecology, WB Saundere Co. Ltd., Philadelphia, 1949.

\bibitem{berec}
L.~Berec, E.~Angulo, F.~Courchamp, Multiple {Allee} effects and population
  management, Trends in Ecology \& Evolution 22 (2007) 185--191.

\bibitem{courchamp}
F.~Courchamp, T.~Clutton-Brock, B.~Grenfell, Inverse density dependence and the
  {Allee} effect, Trends in Ecology \& Evolution 14 (1999) 405--410.

\bibitem{stephens}
P.~Stephens, W.~Sutherland, Consequences of the {Allee} effect for behaviour,
  ecology and conservation, Trends in Ecology \& Evolution 14 (1999) 401--405.

\bibitem{liermann}
M.~Liermann, R.~Hilborn, Depensation: evidence, models and implications, Fish
  and Fisheries 2 (2001) 33--58.

\bibitem{courchamp3}
F.~Courchamp, B.~Grenfell, T.~Clutton-Brock, Impact of natural enemies on
  obligately cooperative breeders, Oikos 91 (2000) 311--322.

\bibitem{courchamp2}
F.~Courchamp, L.~Berec, J.~Gascoigne, Allee effects in ecology and
  conservation, Oxford University Press, 2008.

\bibitem{stephens2}
P.~Stephens, W.~Sutherland, R.~Freckleton, What is the {Allee} effect?, Oikos
  87 (1999) 185--190.

\bibitem{yue2}
Z.~Yue, X.~Wang, H.~Liu, Complex dynamics of a diffusive {Holling--Tanner}
  predator--prey model with the {Allee} effect, Abstract and Applied Analysis
  2013.

\bibitem{arancibia2}
C.~Arancibia-Ibarra, E.~Gonzalez-Olivares, The {Holling--Tanner} model
  considering an alternative food for predator, Proceedings of the 2015
  International Conference on Computational and Mathematical Methods in Science
  and Engineering CMMSE 2015 (2015) 130--141.

\bibitem{blows}
T.~Blows, N.~Lloyd, The number of limit cycles of certain polynomial
  differential equations, Proceedings of the Royal Society of Edinburgh:
  Section A Mathematics 98 (1984) 215--239.

\bibitem{chicone}
C.~Chicone, Ordinary Differential Equations with Applications, Vol.~34 of Texts
  in Applied Mathematics, World Scientific, Springer-Verlag New York, 2006.

\bibitem{dumortier}
F.~Dumortier, J.~Llibre, J.~Art{\'e}s, Qualitative theory of planar
  differential systems, Springer Berlin Heidelberg, Springer-Verlag Berlin
  Heidelberg, 2006.

\bibitem{perko}
L.~Perko, Differential Equations and Dynamical Systems, Springer New York,
  2001.

\bibitem{xiao2}
D.~Xiao, S.~Ruan, {Bogdanov--Takens} bifurcations in predator--prey systems
  with constant rate harvesting, Fields Institute Communications 21 (1999)
  493--506.

\bibitem{matcont}
A.~Dhooge, W.~Govaerts, Y.~Kuznetsov, Matcont: a matlab package for numerical
  bifurcation analysis of odes, ACM Transactions on Mathematical Software
  (TOMS) 29 (2003) 141--164.

\bibitem{gaiko}
V.~Gaiko, Global {Bifurcation} {Theory} and {Hilbert's} {Sixteenth} {Problem},
  Vol. 562 of Mathematics and Its Applications, Springer Science \& Business
  Media, 2013.

\end{thebibliography}

\end{document}